\documentclass[a4paper,11pt]{article}
\usepackage[T1]{fontenc}
\usepackage[utf8]{inputenc}
\usepackage{amsmath,amsthm}
\usepackage[french,english]{babel}
\usepackage[dvips]{graphicx}
\usepackage{mathrsfs}
\usepackage{tikz}
\usetikzlibrary{arrows,positioning,shapes,matrix,backgrounds}
\usetikzlibrary{fit,calc,decorations.pathreplacing}

\usepgflibrary{arrows}
\usepackage[all]{xy}
\usepackage{amsfonts,amssymb}
\usepackage{geometry}
\usepackage{hyperref}
\usepackage{stmaryrd}
\usepackage{fancyhdr}
\usepackage{url}
\usepackage{enumerate}

\theoremstyle{definition}

\newtheorem*{ack}{Acknowledgements}
\theoremstyle{theorem}
\newtheorem*{thm*}{Theorem}

\newtheorem{prop}{Proposition}[section]

\newtheorem{lemma}[prop]{Lemma}

\newtheorem{thm}[prop]{Theorem}

\newtheorem{corollary}[prop]{Corollary}

\newtheoremstyle{pourlesremarques}{\topsep}{\topsep}{\normalfont}{}{\bfseries}{.}{ }{}
\theoremstyle{pourlesremarques}
\newtheorem{rem}[prop]{Remark}
\newtheorem*{rem*}{Remark}
\newtheoremstyle{pourlesexemples}{\topsep}{\topsep}{\normalfont}{}{\bfseries}{.}{ }{}
\theoremstyle{pourlesexemples}

\def\presuper#1#2%
  {\mathop{}%
   \mathopen{\vphantom{#2}}^{#1}%
   \kern-\scriptspace%
   #2}

\newcommand{\Ind}{\operatorname{Ind}}
\newcommand{\ind}{\operatorname{ind}}

\newcommand{\Hom}{\operatorname{Hom}}
\renewcommand{\subset}{\subseteq}

\newcommand{\GL}{\operatorname{GL}}

\renewcommand{\l}{\lambda}

\newcommand{\C}{\mathbb{C}}

\def\rl{\mathrm{r}_{\ell}}

\def\GL{\operatorname{GL}}

\def\\Hom{\operatorname{\Hom}}

\def\dim{\operatorname{dim}}

\def\Ql{\overline{\mathbb{Q}_{\ell}}}

\def\Zl{\overline{\mathbb{Z}_{\ell}}}
\def\Fl{\overline{\mathbb{F}_{\ell}}}

\def\leq{\leqslant}

\def\geq{\geqslant}

\def\presuper#1#2%
  {\mathop{}%
   \mathopen{\vphantom{#2}}^{#1}%
   \kern-\scriptspace%
   #2}
\makeatletter
\DeclareRobustCommand{\rvdots}{%
  \vbox{
    \baselineskip4\p@\lineskiplimit\z@
    \kern-\p@
    \hbox{.}\hbox{.}\hbox{.}  \hbox{.}
  }}
\makeatother
\definecolor{dviolet}{RGB}{100,0,100}
\definecolor{blue}{RGB}{0,0,100}

\title{Whittaker functionals and contragredient in characteristic not $p$}
\author{Nadir Matringe and Justin Trias}
\date{\today}
\begin{document}
\maketitle

\abstract{Let $R$ be an algebraically closed field and $\ell$ be its characteristic. Let $G$ be a locally profinite group having a compact open subgroup of invertible pro-order in $R$. Take $N$ a closed subgroup of $G$ exhausted by compact subgroups of invertible pro-orders in $R$ and fix a smooth character $\theta$ of $N$. For $\pi$ an irreducible smooth $R$-representation of $G$ whose matrix coefficients are compactly supported modulo the center (we call it $Z$-compact), we show that the dimensions $\Hom_{N}(\pi,\theta)$ and $\Hom_{N}(\pi^\vee,\theta^{-1})$ are equal provided one of the two is finite. We derive a few applications from this result. First, we prove that any $G$-intertwiner from $\pi$ to $\Ind_N^G(\theta)$ has image in $\textup{ind}_{ZN}^G(\omega_\pi \theta)$, where $\omega_\pi$ is the central character of $\pi$, and the Whittaker space of $\pi$ agrees with that of its Whittaker periods. Second, it applies to quasi-split groups over non Archimedean local fields of residual characteristic $p \neq \ell$ and where $N$ is the unipotent radical of a Borel subgroup of $G$ together with a generic character $\theta$. Our equality of dimensions turns out to be a good replacement for Rodier's crucial use of complex conjugation in the proof of Whittaker multiplicity at most one for cuspidal representations. Then by a lifting argument, we recover Rodier's generalization of the Gelfand-Kazhdan property for $R$-valued $(\theta^{-1}\otimes \theta)$-equivariant distributions on $G$. This latter fact, together with Rodier's heridity property, which is valid in our context, leads to the multiplicity at most one of Whittaker functionals over $R$. We also give other applications, including a generalization over $R$ of a result for complex representations proved by Chang Yang and initially conjectured by Dipendra Prasad.}

\section{Introduction}

Let $F$ be a non Archimedean local field of residual characteristic $p$ and let $G = \textbf{G}(F)$ be the $F$-points of a quasi-split reductive group $\textbf{G}$ over $F$. Fix a Borel subgroup of $G$ and call $N$ its unipotent radical. Let $R$ be an algebraically closed field of characteristic $\ell \neq p$. Choose as well a non-degenerate character $\theta$ of $N$. Over $R=\C$, the multiplicity at most one of Whittaker functionals proved by Rodier in \cite{Rodier}, following Gelfand-Kazhdan \cite{GK} for $G=\GL_n(F)$ asserts that for all irreducible (admissible) complex representations $\pi$ of $G$, the dimension of $\textup{Hom}_N(\pi,\theta)$ satisfies:
$$d_\theta(\pi) \in \{ 0 , 1 \}.$$
The proof of this statement can be divided into three intermediate steps. Writing $\pi^\vee$ for the contragredient of $\pi$, it breaks down as:
\begin{enumerate}
\item[(1)] proving an inequality $d_\theta(\pi) \times d_{\theta^{-1}}(\pi^\vee) \leq 1$ for all irreducible $\pi$;
\item[(2)] showing that $d_\theta(\pi) = d_{\theta^{-1}}(\pi^\vee)$ when $\pi$ is cuspidal;
\item[(3)] using the so-called heredity property to extend the above equality to all irreducible representations. 
\end{enumerate}
Note that in \cite[Lem 2]{PrasMVW} Prasad proves (2) for all irreducible representations by reducing to the tempered case, but in all cases the arguments rely on the fact that for unitary representations one has $\pi^\vee \simeq {}^c \pi$ where ${}^c \pi$ is the complex conjugate of $\pi$. Clearly this latter argument depends on specific properties of $\C$. However, Point (3) generalizes in a rather straightforward way to any coefficient field, and so does Point (1) as long as the characteristic of $R$ is not $2$. In order to avoid this annoying restriction on the characteristic, we prefer to deduce Point (1) by lifting $(\theta^{-1} \otimes \theta)$-equivariant distributions from positive characteristic fields to characteristic zero ones. Let $I_\theta$ be the Rodier involution of $G$ defined in Section \ref{SecM1}, which in the case of $\GL_n(F)$ is the transpose inverse composed with the conjugation by the antidiagonal matrix with all entries equal to one on the second diagonal. Our method, based on the rather general lifting result from Lemma \ref{lemma lifting of equivariant forms} from residue fields to discrete valuation rings, allows us to prove in Proposition \ref{invariance_Atheta_prop} that any $(\theta^{-1} \otimes \theta)$-equivariant distributions on $G$ with values in $R$ is invariant under the Rodier anti-involution $g\mapsto I_\theta(g^{-1})$.

The main obstacle we need to overcome is point (2). Our argument is elementary and possibly well known to some experts. It is valid in the much more general setting of locally profinite groups and applies to $Z$-compact irreducible representations $\pi$ of such groups $G$ admitting a compact open subgroup of invertible pro-order in $R$. Taking $N$ a closed subgroup of $G$ exhausted by compact subgroups of invertible pro-orders in $R$ and fixing a smooth character $\theta$ of $N$, we prove that the dimension $d_\theta(\pi)$ of $\Hom_{N}(\pi,\theta)$ is finite if and only if the dimension $d_{\theta^{-1}}(\pi^\vee)$ of $\Hom_{N}(\pi^\vee,\theta^{-1})$ is finite as well, in which case:
\[d_\theta(\pi) = d_{\theta^{-1}}(\pi^\vee).\]
For quasi-split groups $G=\textbf{G}(F)$ as above, an irreducible cuspidal representation is always $Z$-compact, so point (2) is a direct consequence of our result.

We derive two main applications from our result. First, Chang Yang \cite{Yang} recently proved a conjecture of Dipendra Prasad \cite{PrasMVW} saying that, for $G$ quasi-split as above, all irreducible complex $\theta$-generic representations $\pi$ of $G$ satisy $\pi^\vee\simeq \pi \circ I_{\theta}$. The main argument in the proof  is generalizing Gelfand and Kazhdan \cite{GK} and is a consequence of the invariance of $(\theta^{-1} \otimes_R \theta)$-equivariant distibutions by $I_\theta$. In the course of generalizing Rodier's result on these distributions, we remarked that Yang's proof is valid over $R$. As a second application, we simply observe that our groups $G$ are quite general: in Section \ref{Sec Met} we replace them by Kazhdan-Patterson metaplectic covers of $\GL_n(F)$ to obtain some information about the multiplicity of their Whittaker models. Nevertheless, our Section \ref{Sec other approach} explains how a more direct approach is also fruitful in the context of $\ell$-modular representations \textit{i.e.} when $R=\Fl$. In particular we recover most of the applications above by a simple reduction argument from $\Ql$-coefficients \linebreak to $\Fl$-coefficients. The only drawback here comes from the use of Proposition \ref{prop DHKM} which is proved in the forthcoming work \cite{DHKMbanal} and relies on techniques that extend far beyond the usual representation theoretic methods as it requires some deep results of Fargues and Scholze.

\ack We thank Jean-François Dat, David Helm, Rob Kurinczuk, and Shaun Stevens for useful conversations and correspondence. Also we are greatly indebted to Guy Henniart for reading our first draft and pointing out some problems, as well as for giving us the most general case of Lemma \ref{coinvariants_regular_rep_lemma}. Finally we thank the referee for very their accurate reading, and their useful comments and suggestions, leading to the clarification of some arguments. The first named author would like to thank the CNRS and Imperial college London for giving him respectively a délégation and an ICL-CNRS felowship in 2022, and the Erwin-Schrödinger institute for funding the Research in Teams Project: l-modular Langlands Quotient Theorem and Applications. This work benefited from the hospitality of the aforementioned institutions, as well as from IMJ-PRG where the second author is grateful to have been invited.

\section{Preliminary results} \label{Sec preliminary}

Let $R$ be a commutative unitary ring. For a locally profinite group $G$, we denote by $\textup{Rep}_R(G)$ the category of smooth $R[G]$-modules. When $H$ is a closed subgroup of the locally profinite group $G$, we have the two (non-normalized) induction functors $\textup{ind}_H^G$ and $\textup{Ind}_H^G$ from $\textup{Rep}_R(H)$ to $\textup{Rep}_R(G)$. Let $\chi : H \to R^\times$ be a smooth character. 

\paragraph{Coinvariants.} The functor of $\chi$-coinvariants associates to $V \in \textup{Rep}_R(H)$ the module: 
\[V_\chi = V / V[\chi]\in \textup{Rep}_R(H)\] where: \[V[\chi] = \langle h \cdot v - \chi(h) v \ | \ h \in H \textup{ and } v \in V \rangle.\] 

Let $R'$ be a commutative $R$-algebra, and $\chi_{R'}$ be the composition of $\chi$ with the ring morphism $R \to R'$. The $\chi$-coinvariants fit into an exact sequence $0\rightarrow V [\chi] \rightarrow V \rightarrow V_\chi \rightarrow 0$ and applying the functor $- \otimes_R R'$ to it, we obtain $V[\chi] \otimes_R R' \to V \otimes_R R' \to V_\chi \otimes_R R' \to 0$ by right exactness. The image of $V[\chi] \otimes_R R' \to V \otimes_R R'$ actually is $(V \otimes_R R')[\chi_{R'}]$ so:

\begin{lemma}\label{lemme Justin}
Let $R'$ be a commutative $R$-algebra. Then the map $V \otimes_R R' \to V_\chi \otimes_R R'$ above induces a canonical identification between $(V \otimes_R R')_{\chi_{R'}}$ and $V_\chi \otimes_R R'$. \end{lemma}

The following classical fact \cite[I.4.11]{Vbook} asserts that it is exact in an important situation:

\begin{lemma}\label{lemme exact2}
When $H = N$ is exhausted by compact subgroups of invertible pro-order in $R$, the functor of $\chi$-coinvariants $V \mapsto V_\chi$ is exact. \end{lemma}

The proof of Lemma \ref{lemme exact2} is a consequence of the following description of $V[\chi]$ due to Jacquet and Langlands: a vector $v$ belongs to 
$V[\chi]$ if and only if there is a compact open subgroup $K$ of $N$ such that $e_{\chi,K}(v)=0$ where:
\[e_{\chi,K}(v)=\int_K \chi^{-1}(n) \ (n\cdot v)  \ dn.\]
From this description and Lemma \ref{lemme Justin}, we deduce a useful lifting result for $\chi$-equivariant linear forms. First we state the following preliminary general lemma, the proof of which benefited from clarifications of Guy Henniart.

\begin{lemma}\label{lemma decomposition}
Let $M$ be a module over an integral domain $\mathcal{B}$ with fraction field $\mathcal{L}$. Then it has a decomposition $M = M_1 \oplus M_2$, where $M_1$ is the maximal $\mathcal{L}$-vector space contained in $M$, and $M_2$ does not contain any $\mathcal{L}$-line.
\end{lemma}
\begin{proof}
First the existence of $M_1$ is given by Zorn's Lemma. Then we have the chain $\Hom_{\mathcal{B}}(-,M_1)\simeq \Hom_\mathcal{B}(-,\Hom_{\mathcal{L}}(\mathcal{L},M_1))\simeq 
\Hom_\mathcal{L}(-\otimes_{\mathcal{B}}\mathcal{L}, M_1)$, where the second isomorphism is the tensor-hom adjunction whereas the first 
holds because $\Hom_{\mathcal{L}}(\mathcal{L},-)$ identifies with the identity functor on the category of $\mathcal{L}$-vector spaces. Now using the general property of fraction fields that 
$\mathcal{L}$ is $\mathcal{B}$-flat, and that $M_1$ is injective as an $\mathcal{L}$-vector space, we deduce that the functor $\Hom_{\mathcal{B}}(-,M_1)$ is exact. In particular $M_1$ is injective over $\mathcal{B}$, hence it has a complement $M_2$ inside $M$. If such a complement was to  contain a line $D$, then 
$M_1\oplus D$ would be an $\mathcal{L}$-vector space contained in $M$, contradicting the maximality of $M_1$. \end{proof}

We can now prove our lifting result. 

\begin{lemma}\label{lemma lifting of equivariant forms}
Let $\mathcal{A}$ be a complete discrete valuation ring, denote by $\mathcal{K}$ its fraction field and by $k$ its residue field. Assume that $H = N$ is exhausted by compact subgroups of invertible pro-order in $\mathcal{A}$ and let $V \in \textup{Rep}_{\mathcal{A}}(N)$ be a torsion free $\mathcal{A}$-module such that $V \otimes_\mathcal{A} \mathcal{K}$ has countable dimension. Then for all $T \in \textup{Hom}_N(V \otimes_{\mathcal{A}} k , \chi_k)$ there exists $\tilde{T} \in \textup{Hom}_N(V,\chi)$ reducing to $T$ over $k$.
\end{lemma}

\begin{proof} Set $V^\mathcal{K} = V \otimes_\mathcal{A} \mathcal{K}$ and $V^k = V \otimes_\mathcal{A} k$. As $V$ is torsion free over $\mathcal{A}$, it embeds as an $\mathcal{A}$-submodule of $V^\mathcal{K}$ via the map $v \mapsto v \otimes_\mathcal{A} 1$. We claim that $V[\chi]=V \cap V^\mathcal{K}[\chi_\mathcal{K}]$: the inclusion $V[\chi]\subseteq V \cap V^\mathcal{K}[\chi_\mathcal{K}]$ is obvious, and if $v\in V \cap V^\mathcal{K}[\chi_\mathcal{K}]$ it satisfies that $e_{\chi,K}(v)= e_{\chi_\mathcal{K},K}(v)=0$ for some large enough $K$, and we conclude 
$v\in V[\chi]$. Hence the injection of $V$ into $V^\mathcal{K}$ induces an embedding $V_\chi \hookrightarrow V^\mathcal{K}_{\chi_\mathcal{K}}$. We can now apply Lemma \ref{lemma decomposition} to $M=V_{\chi}$ and $\mathcal{B}=\mathcal{A}$, giving rise to a decomposition 
$V_{\chi}=V_{\chi,1}\oplus V_{\chi,2}$. Thus the $\mathcal{A}$-module $V_{\chi,2}$ does not contain any $\mathcal{K}$-line and $V_{\chi,2} \otimes_\mathcal{A} \mathcal{K}$ is of countable dimension by hypothesis. Because $\mathcal{A}$ is principal, local and complete, we deduce from \cite[I.9.2 \& App C.5]{Vbook} that $V_{\chi,2}$ is free over $\mathcal{A}$. Moreover 
$V_{\chi}\otimes_{\mathcal{A}} k= V_{\chi,2}\otimes_{\mathcal{A}} k$ because $V_{\chi,1}\otimes_{\mathcal{A}} k=0$ and the tensor product preserves direct sums. The linear form $T'$ in $\textup{Hom}_k(V_{\chi_k}^k ,k)$ corresponding to $T$ via the canonical isomorphism $\textup{Hom}_N(V^k ,\chi_k)\simeq \textup{Hom}_k(V_{\chi_k}^k ,k)$ lifts to an element $\tilde{T}'_2\in \textup{Hom}_\mathcal{A}(V_{\chi,2} ,\mathcal{A})$ because $V_{\chi,2}$ is free over $\mathcal{A}$. However $\textup{Hom}_\mathcal{A}(V_{\chi,2} ,\mathcal{A})$ identifies with the $\mathcal{A}$-submodule of $\textup{Hom}_\mathcal{A}(V_{\chi} ,\mathcal{A})$ consisting of linear forms vanishing on $V_{\chi,1}$, so we can see $\tilde{T}'_2$ as a linear form $\tilde{T}' \in\textup{Hom}_\mathcal{A}(V_{\chi} ,\mathcal{A})$. Finally we just take $\tilde{T}\in \textup{Hom}_N(V,\chi)$ to be the equivariant linear map corresponding to $\tilde{T}'$ via the canonical identification $\textup{Hom}_N(V,\chi)\simeq \textup{Hom}_\mathcal{A}(V_{\chi} ,\mathcal{A})$. \end{proof}

\paragraph{Coinvariants of the regular representation.}  We denote by $C_c^\infty(G)$ the space of locally constant and compactly supported functions on $G$ with values in $R$. It is a smooth $(G\times G)$-module for the action $((\l\otimes \rho) (g,g')f)(x)=f(g^{-1}xg')$. Let $H$ be a \textit{unimodular} closed subgroup of $G$ containing an open compact subgroup of invertible pro-order in $R$, so that there exists a right invariant Haar measure $\mu$ on $H$ with values in $R$. In general all our Haar measures are assumed to be normalized on a compact open subgroup. Let $\chi : H \to R^\times$ be a smooth character. 

\begin{lemma} \label{coinvariants_regular_rep_lemma} 
The map:
\[\begin{array}{cccc} p_{H,\chi} :& C_c^\infty(G) & \rightarrow & \textup{ind}_H^G(\chi) \\
 & f & \mapsto & p_{H,\chi}f \end{array} \textup{ where } p_{H,\chi}f(g) = \int_H \chi^{-1}(h) f(h g) dh\]
defines a surjective $(H\times G)$-equivariant map from $C_c^\infty(G)$ to $\textup{ind}_H^G(\chi)$. 
Furthemore, this morphism factors through an $(H\times G)$-equivariant isomorphism:
$$C_c^\infty(G)_{\lambda,\chi^{-1}}\simeq \textup{ind}_H^G(\chi)$$
where the index indicates that the coinvariants are taken with respect to the action of $\lambda$. \end{lemma}
\begin{proof}
The first part is proved in \cite[App B.1]{AKMSS}, and when $R$ is a field the second part also follows easily from \cite[App B.1]{AKMSS} using the fact that linear forms separate points of a vector space. However the following short argument works in full generality. It was given to us by Guy Henniart. First observe that in the proof of \cite[Lem B.4 (i)]{AKMSS}, the linear form $\mathscr{L}$ can be replaced by any $R$-linear map with values in an $R$-module $M$. Taking $M = C_c^\infty(G)_{\lambda,\chi^{-1}}$ shows that the kernel of $p_{H,\chi}$ is equal to $C_c^\infty(G)_{\l}[\chi]$ and the result follows. 
\end{proof}

As an immediate consequence of this result we deduce the second part of the following corollary, the first one being the integration in stages formula:

\begin{corollary} \label{factorizing_coinvaiants_cor} If $H'$ is a closed subgroup of $H$, together with their respective Haar measures, there exists a (unique) Haar measure on $H' \backslash H$, such that the map $p_{H,\chi}:C_c^\infty(G) \twoheadrightarrow \textup{ind}_H^G(\chi)$ factors through:
$$\begin{array}{cccc} q_{H,H'} :& \textup{ind}_{H'}^G( \chi) &  \twoheadrightarrow & \textup{ind}_H^G(\chi) \\
 & f & \mapsto & \displaystyle \int_{H'\backslash H} \chi^{-1}(h) p_{H,\chi}f(h g) d h \end{array}.$$
In particular $p_{H,\chi}= q_{H,H'} \circ p_{H',\chi}$ and $q_{H,H'}$ is $(H' \times G)$-equivariant. Moreover if $H'$ is normal in $G$, then $\textup{ind}_{H'}^G( \chi)$ is naturally a left $H$-module and the morphism $q_{H,H'}$ induces an $(H \times G)$-isomorphism $\textup{ind}_{H'}^G(\chi)_{\lambda,\chi^{-1}} \simeq \textup{ind}_H^G(\chi)$.
  \end{corollary}

\section{Duality and Whittaker functionals}

Let $G$ be a locally profinite group with center $Z$. Let $R$ be an algebraiclly closed field such that $G$ has an open compact subgroup of invertible pro-order in $R$. We recall that an irreducible $\pi$ in $\textup{Rep}_R(G)$ is said to be $Z$-compact if all of its matrix coefficients have compact support modulo $Z$. Such representations are admissible and admit central characters \linebreak \cite[I.7.11 Prop]{Vbook}. In particular the contragredient $\pi^\vee$ of an irreducible $Z$-compact representation $\pi$ is itself $Z$-compact and irreducible. Denoting by $\omega_\pi$ the central character of $\pi$, we have a $(G\times G)$-morphism for the action induced by $\lambda\otimes \rho$ on $\ind_Z^G(\omega_\pi)$:
$$\pi^\vee \otimes_R \pi \lhook\joinrel\xrightarrow{\hspace{0.35cm} c \hspace{0.35cm}} \textup{ind}_Z^G(\omega_\pi)$$
where the image $c(v^\vee \otimes_R v)$ is the coefficient $g \mapsto v^\vee(\pi(g) v)$ of $\pi$. The representation on the left-hand side is an irreducible $(G\times G)$-module by a straightforward generalization of \cite[Th 1]{flath} to our setting, so this arrow is indeed injective. 

\begin{prop}\label{prop main} Let $\pi$ be an irreducible $Z$-compact representation in $\textup{Rep}_R(G)$ and $\theta$ be a character of a subgroup $N$ exhausted by compact subgroups of invertible pro-orders in $R$. Assume that $\omega_\pi$ agrees with $\theta$ on $Z \cap N$ so that $[\omega_\pi \theta] ( zn ) = \omega_\pi(z) \theta(n)$ is a character of~$ZN$. Then the coefficient map above induces an $(N \times G)$-morphim:
\[(\pi^\vee)_{\theta^{-1}} \otimes_R \pi \lhook\joinrel\xrightarrow{\hspace{0.25cm} c_{\theta^{-1}} \hspace{0.05cm}} \textup{ind}_{ZN}^G ([\omega_\pi \theta]).\] \end{prop}

\begin{proof} Note that $Z$ is a normal subgroup of $ZN$ and that both groups are unimodular. As a result of Corollary \ref{factorizing_coinvaiants_cor}, we have a surjective $(ZN \times G)$-morphism
$\textup{ind}_Z^G(\omega_\pi) \twoheadrightarrow \textup{ind}_{ZN}^G([\omega_\pi \theta])$ identifying with $\theta^{-1}$-coinvariants for $\lambda$. By Lemma \ref{lemme exact2}, the $\theta$-coinvariants are exact so we have an injection $(\pi^\vee)_{\theta^{-1}} \otimes_R \pi \hookrightarrow \textup{ind}_{ZN}^G([\omega_\pi \theta])$. \end{proof}

\begin{rem}
When $G$ is reductive, $N$ is the unipotent radical of a minimal parabolic subgroup of $G$, and $\theta$ is non-degenerate, the inclusion 
$\pi\hookrightarrow \textup{ind}_{ZN}^G ([\omega_\pi \theta])$ when $\pi$ is cuspidal is proved for complex representations in \cite[Cor 6.5]{CS}, using a different method. In this situation, a converse statement also holds thanks to \cite[Th 10]{DelormePWi}.
\end{rem}

Our main result on Whittaker multiplicities is a simple consequence of this result. Note that $\dim(\pi_\theta)$ and $\dim(\Hom_N(\pi,\theta))$ are equal when one of the two is finite. Set:
\[d_\theta(\pi) := \dim(\Hom_N(\pi,\theta)).\]

\begin{thm}\label{multiplicities_main_thm} Let $\pi \in \textup{Rep}_R(G)$ be irreducible $Z$-compact. Then $d_{\theta^{-1}}(\pi^\vee)$ is finite if and only if $d_{\theta}(\pi)$ is finite as well, in which case: 
$$d_{\theta^{-1}}(\pi^\vee) = d_\theta(\pi).$$\end{thm}

\begin{proof} The vector space $\textup{Hom}_G(\pi,\textup{Ind}_N^G(\theta))$ has dimension $d_{\theta}(\pi)$ by Frobenius reciprocity, i.e. $d_\theta(\pi)=\dim(\pi_\theta)$, but as it contains $\textup{Hom}_G(\pi,\textup{ind}_{ZN}^G ([\omega_\pi \theta]))$ its dimension is at least $\dim((\pi^\vee)_{\theta^{-1}})$ thanks to Proposition \ref{prop main}. So $d_\theta(\pi) \geq \dim((\pi^\vee)_{\theta^{-1}})$. Now a similar reverse inequality holds because $\pi$ is admissible so $(\pi^\vee)^\vee \simeq \pi$. It gives $d_{\theta^{-1}}(\pi^\vee) \geq \dim(\pi_\theta)$, where by definition  $d_{\theta^{-1}}(\pi^\vee):= \dim(\Hom_N(\pi^\vee,\theta^{-1}))=\dim((\pi^\vee)_{\theta^{-1}})$. \end{proof}

As a reminder of the proof of Propisition \ref{prop main}, we sketch the commutative diagram:
\[\xymatrixcolsep{5pc}\xymatrix{
\pi^\vee \otimes_R \pi \ar@{->>}[d] \ar@{^{(}->}[r]^{c} & \textup{ind}_Z^G(\omega_\pi) \ar@{->>}[d]  & \\
\pi^\vee_{\theta^{-1}} \otimes_R \pi  \ar@{^{(}->}[r]^{c_{\theta^{-1}}} & \textup{ind}_{ZN}^G([\omega_\pi \theta]) \ar@{^{(}->}[r] & 
\textup{Ind}_N^G( \theta) }\]
where according to Corollary \ref{factorizing_coinvaiants_cor} the composition:
\[\pi^\vee \otimes_R \pi \longrightarrow \textup{ind}_{ZN}^G([\omega_\pi \theta])\]
actually is the map associating to $v^\vee \otimes v$ the function:
\[W_{v^\vee,v}:g \mapsto \int_{Z \backslash ZN} \theta^{-1}(n) c_{v^\vee,v}(n g) dn.\] Suppose moreover that $d_\theta(\pi)$ is finite, then Theorem 3.2 implies that the image of the composition of the two bottom arrows is the $\pi$-isotypic component of $\textup{Ind}_{N}^G(\theta)$. Hence:
\[\Hom_G(\pi,\Ind_N^G(\theta))=\Hom_G(\pi,\ind_{ZN}^G([\omega_\pi \theta])=\{v\mapsto W_{v^\vee,v}, v^\vee \in V^\vee\}.\]

Define for $v^\vee\in \pi^\vee$ the Whittaker period $W_{v^\vee} \in \textup{Hom}_N(\pi,\theta)$ on $\pi$ attached to $v^\vee$: 
$$W_{v^\vee} : v \mapsto \int_{Z \backslash ZN} \theta^{-1}(n) c_{v^\vee,v}(n) dn.$$

A consquence of the above discussion and of Frobenius reciprocity is the following.

\begin{corollary}\label{cor main} Let $\pi$ be an irreducible $Z$-compact admissible representation of $G$. Suppose that $d_\theta(\pi)$ is finite. Then any Whittaker model of $\pi$ lies inside $\textup{ind}_{ZN}^G([\omega_\pi \theta])$ and moreover the Whittaker space of $\pi$ is equal to that of its Whittaker periods:
\[\Hom_N(\pi,\theta)=\{W_{v^\vee} \ | \  \ v^\vee \in \pi^\vee\}.\]
\end{corollary}

\section{Multiplicity at most one for Whittaker functionals}\label{SecM1}

Now let $F$ be a non Archimedean local field of residual characteristic $p$ and $G$ be the group of $F$-points of a quasisplit reductive group over $F$. We moreover suppose that $R$ is algebraically closed of characteristic $\ell\neq p$. We fix a maximal $F$-split torus $T$ of $G$, as well as a Borel subgroup $B$ containing $T$ with unipotent radical $N$. This in particular fixes a set of simple roots $\Delta$ of $T$ and a longest Weyl element $w_0$ in the Weyl group $W(G,T)$. We fix a non-trivial character $\psi:F\rightarrow R$, as well as a non-degenerate character $\theta:N\rightarrow F$. This means that for all $\alpha\in \Delta$ the restriction of $\theta$ to the root subgroup of $\alpha$ is non-trivial. Let $N^{\textup{ab}}$ be the abelianization of $N$. We have a canonical isomorphism $\Hom(N,R^\times)\simeq \Hom(N^{\textup{ab}},R^\times)$ for smooth characters, which we denote by $\mu \mapsto \overline{\mu}$. Now by \cite{BHderived} we can endow $N^{\textup{ab}}$ with a structure of finite dimensional $F$-vector space, with basis $\Delta$. Hence by Pontryagin duality the map $\overline{\Phi}\mapsto \psi\circ \overline{\Phi}$ gives an isomorphism $\Hom_F(N^{\textup{ab}},F)\simeq \Hom(N^{\textup{ab}},R^\times)$. In particular $\overline{\theta}=\psi\circ \overline{\Theta}$ for a unique $\overline{\Theta}\in \Hom_F(N^{\textup{ab}},F)$, and we set $\Theta$ to be its composition with 
$N \to N^{\textup{ab}}$. According to \cite[Prop 5]{Rodier}, there exists a unique automorphism $I_\theta$ of the group $G$ satisfying:

\begin{itemize}
\item $I_\theta(N)=N$;
\item $\Theta\circ I_{\theta}=\Theta^{-1}$;
\item $I_\theta(t)=w_0 t^{-1} w_0$ for all $t\in T$. 
\end{itemize}

\noindent In addition this is an involution \textit{i.e.} $I_\theta^2 = \textup{Id}$. Let $A_\theta$ be the anti-involution of $G$ defined by: 
\[A_{\theta}(g) = I_{\theta}(g^{-1}) \textup{ for } g \in G.\]

\paragraph{Uniqueness of Whittaker models.} The main goal of this section is to explain how Rodier's proof can be generalized to representations with coefficients in $R$. We show that points (1)-(2)-(3) from the introduction hold over $R$. This is also known as the uniqueness of Whittaker models:
\begin{thm}\label{Thm M1}
Let $\pi \in \textup{Rep}_R(G)$ be irreducible. Then: 
$$d_\theta(\pi) \leq 1.$$ \end{thm}

\subsection{Inequality (1)}

We generalize the key result \cite[Prop 9]{Rodier} over $R$. 

\begin{prop} \label{invariance_Atheta_prop} All $\theta^{-1}\otimes \theta$-equivariant distributions on $C_c^\infty(G)$ are invariant under $A_{\theta}$. \end{prop}

\begin{proof}
Over $R=\C$ this is \cite[Prop 9]{Rodier}. The proof of Rodier is valid in characteristic zero, and in fact in characteristic different from $2$ (the argument uses at the end that $T=-T$ implies $T=0$ for a distribution $T$). Hence to deduce the result above from the characteristic zero case, we use the lifting result of Lemma \ref{lemma lifting of equivariant forms} with $k=R$, its Witt vectors $\mathcal{A}=W(R)$, the representation $V=C_c^\infty(G,W(R))$, the group $H=N\times N$ and the character $\chi=\tilde{\theta}^{-1}\otimes \tilde{\theta}$ for $\tilde{\theta}$ lifting $\theta$ to $W(R)$. Note that $V$ is torsion free as an $A$-module \cite[App C.5]{Vbook} and is contained in $C_c^\infty(G,\mathcal{K})$ whose dimension is countable over $\mathcal{K}=\textup{Frac}(\mathcal{A})$. Also a distribution $\tilde{T} \in \textup{Hom}_N(V,\tilde{\theta}^{-1} \otimes \tilde{\theta})$ naturally extend to a distribution over $\mathcal{K}$, implying the $A_\theta$-invariance of $\tilde{T}$ and its reduction as well. \end{proof}
The proofs of \cite[Prop 10 \& Prop 11]{Rodier} then apply verbatim over $R$ and Rodier deduces following the argument of Gelfand and Kazhdan (\cite{GK}) the inequality below for complex representations \cite[Prop 11]{Rodier}. Note that by \cite{Vbook}, irreducible representations of $G$ are always admissible so we remove the term admissible from the statement below. 

\begin{thm} \label{cor ineg} Let $\pi \in \textup{Rep}_R(G)$ be irreducible. Then:
\[d_\theta(\pi) \times d_{\theta^{-1}}(\pi^\vee) \leq 1.\] \end{thm}

\subsection{Equality (2) in the cuspidal case}

As mentioned in the introduction, the contragredient $\pi^\vee$ of a irreducible complex unitary representation $\pi$ is equal to its complex conjugate: $\pi^\vee \simeq {}^c \pi$. Complex cuspidal representations are unitarisable \textit{i.e.} unitary up to a twist by a character, and both \cite[Prop 13]{Rodier} and \cite[Lem 2]{PrasMVW} rely on this specific property of unitary representations to prove $d_\theta(\pi) = d_{\theta^{-1}}(\pi^\vee)$ for irreducible complex representations. Our Theorem \ref{multiplicities_main_thm} overcomes the absence of this unitarian trick in the setting of cuspidal (or equivalently $Z$-compact \cite[II.2.7]{Vbook}) $R$-representations:

\begin{thm} \label{multiplicity_equal_whittaker_section_thm}
Let $\pi \in \textup{Rep}_R(G)$ be irreducible cuspidal. Then:
$$d_\theta(\pi) = d_{\theta^{-1}}(\pi^\vee).$$ \end{thm}

\subsection{Rodier's heredity property (3)}

Let $P = MU$ be a standard parabolic subgroup of $G$ and let $\tau \in \textup{Rep}_R(M)$ be irreducible. We denote by $i_P^G$ the parabolic induction (non-normalized) functor. Set $N_M = M \cap w_0 N w_0^{-1}$ and define a character $\theta_M$ by $\theta_M (n) = \theta(w_0^{-1} n w_0)$ for $n \in N_M$. Rodier proves the following in \cite[Sec 5]{Rodier}, and its proof is valid for $R$-representations thanks to \cite[App B.1]{AKMSS} again.

\begin{thm} \label{inequality_whittaker_section_thm}  
We have:
$$\textup{Hom}_N(\textup{ind}_P^G(\tau),\theta) \simeq \textup{Hom}_{N_M}(\tau,\theta_M).$$
In particular:
$$d_\theta(\ind_P^G(\tau)) = \sum m_\pi \ d_\theta(\pi) = d_{\theta_M}(\tau)$$
where $\sum m_\pi \pi$ is the semisimplification of $i_P^G(\tau)$. \end{thm}

\begin{rem}
Rodier uses normalized parabolic induction, but because an unramified character of $P$ must be trivial on $N_M$, one deduces the above statement from that in \cite{Rodier}. Moreover up notational changes, one can read Rodier's proof for the non-normalized version of parabolic induction. 
\end{rem}

The existence of the cuspidal support \cite[II.2.4]{Vbook} asserts that all irreducible representations embed into some $\textup{ind}_P^G(\tau)$ with $\tau \in \textup{Rep}_R(M)$ irreducible cuspidal. By the theorem above, as well as Theorems \ref{multiplicity_equal_whittaker_section_thm} and \ref{inequality_whittaker_section_thm} applied to $\tau$, we obtain:
$$d_\theta(\pi) \leq d_{\theta_M}(\tau) \leq 1.$$

\section{Duality and the Rodier involution}\label{sec Yang}

For complex representations, one has the following result conjectured by Dipendra Prasad in \cite[Conj 1]{PrasMVW} and proved by Chang Yang in \cite[Cor 7.7]{Yang} following the method of Gelfand and Kazhdan \cite[Th 4 b)]{GK}. Recall that in the context of Section \ref{SecM1}, the involution $I_\theta$ is an automotphism lifting the Chevalley involution of $G$. Yang assumes his field $F$ to be $p$-adic but its proof works for any non Archimedean local fields in view of Proposition \ref{invariance_Atheta_prop}, which alternatively is \cite[Th 7.6]{Yang}. Our proposition being valid over any algebraically closed fields of characteristic not $p$, we recover Yang's result over such fields as the rest of his proof applies verbatim in our context. Hence we deduce:

\begin{prop}\label{prop Yang}
For all irreducible $\theta$-generic representations $\pi \in \textup{Rep}_R(G)$, we have: 
$$\pi^\vee\simeq \pi\circ I_{\theta}.$$ \end{prop}

\section{An application to the metaplectic cover of $\GL_n(F)$}\label{Sec Met}

We let $r$ be a positive integer such that $F$ 
contains all $r$-th roots of unity and denote by $\mu_r(F)$ this set of roots. In \cite{KP}, Kazhdan and Patterson defined $r$-fold covers of $G_n = \GL_n(F)$ as a group extension of $G_n$ by $\mu_r(F)$. Let $N$ be the unipotent radical of the group of upper triangular matrices $B$ in $G_n$ and $\theta$ be a non-degenerate character of $N$.

Let $\widetilde{G}_n$ be such an $r$-fold cover. Following \cite{Banks}, we can embed $N$ as a subgroup of $\widetilde{G}_n$. Therefore the notion of Whittaker models for $\theta$ also makes sense for irreducible representations of $r$-fold covers after specifying an embedding of $N$. Parabolic subgroups also make sense in this context, so does parabolic induction. Then the main result of \cite{Banks}, which is the extension to $\widetilde{G}_n$ of Rodier's heredity property for $G_n$, namely Theorem \ref{inequality_whittaker_section_thm} for $G_n$, holds over $R$. Hence an immediate application of Theorem \ref{multiplicities_main_thm} leads to:

\begin{prop}
Let $\pi = \ind_{\widetilde{P}}^{\widetilde{G}_n}(\tau)$ be a representation of $\widetilde{G}_n$ which is parabolically induced from an irreducible cuspidal representation $\tau$. Then we have: 
$$d_\theta ( \pi ) = d_{\theta^{-1}} ( \pi^\vee ),$$ 
where both dimensions are known to be finite thanks to \textup{\cite[Th]{Banks}} and \textup{\cite[Th 1.5.2]{KP}}. 
\end{prop}

We could not find a proof of this fact in the literature  and perhaps this is well known to some experts. Note that this proposition would follow from the beautiful formula conjectured and proved in some instances in \cite{Jiandi} for the Whittaker multiplicity of discrete series of these covering groups. Another interesting question would be to determine whether this result holds for all irreducible representations.

\section{Shortcuts for $\ell$-modular representations} \label{Sec other approach}

In this last section we quickly recover by different means some of our main results from Sections \ref{SecM1} and \ref{sec Yang} in the setting of $\ell$-modular representations \textit{i.e.} when $R=\Fl$. This section bypasses any consideration coming from Theorem \ref{multiplicities_main_thm} and Proposition \ref{invariance_Atheta_prop}. However we have to use a result from Dat, Helm, Kurinczuk and Moss, which has yet to be published. It appears as Proposition \ref{prop DHKM} below and actually relies on deep results of Fargues and Scholze.

We fix a prime number $\ell$ different from $p$ and let $G$ be as in Section \ref{SecM1}. We call $\ell$-adic a representation of $G$ over $\Ql$ and $\ell$-modular a representation of $G$ over $\Fl$. Note that fixing an isomorphism $\Ql\simeq \C$, the theory of smooth complex and $\ell$-adic representations are the same. We say that a smooth admissible $\ell$-adic representation $\pi$ of $G$ is integral \cite[I.9.6]{Vbook} if it contains an admissible $\Zl$-lattice $L$ inside the space of $\pi$ which is $G$-stable. If $\pi$ is integral as above, in addition to being finite length, then it is known \cite[I.9.6]{Vbook} that $L\otimes \Fl$ has finite length and that the semi-simplification of $L\otimes \Fl$ is independent of the lattice $L$. In this case we denote by $\rl(\pi)$ this semi-simplification. We recall that a representation of $G$ is called supercuspidal if it does not appear as a subquotient of a representation parabolically induced from a proper parabolic subgroup of $G$. For $\ell$-modular representations, cuspidal representations are not always supercuspidal, and though irreducible representations of $G$ all admit a supercuspidal support, it might not be unique \cite{DatSCS}. We first state a result from the forthcoming paper \cite{DHKMbanal}. 

\begin{prop}\label{prop DHKM}
Let $\overline{\pi}$ be an $\ell$-modular irreducible supercuspidal representation of $G$, then there is an  irreducible integral $\ell$-adic supercuspidal representation of $G$ such that $\rl(\pi)$ contains $\overline{\pi}$. 
\end{prop}

We fix $\theta:N\rightarrow \Zl^\times$ a character of $N$, and denote by $\overline{\theta}$ its composition with the canonical surjection $\Zl\rightarrow \Fl$. We say that an irreducible $\ell$-modular (resp. $\ell$-adic) representation $\overline{\pi}$ (resp. $\pi$) of $G$ is $\overline{\theta}$-generic (resp. $\theta$-generic) if 
$\Hom_N(\overline{\pi},\overline{\theta})\neq \{0\}$ (resp. $\Hom_N(\pi,\theta)\neq \{0\}$). Thanks to Proposition \ref{prop DHKM}, we have:

\begin{corollary}\label{cor DHKM}
Let $\overline{\pi}$ be an $\ell$-modular $\overline{\theta}$-generic representation of $G$, then there is an irreducible $\theta$-generic integral $\ell$-adic representation $\pi$ of $G$ such that $\overline{\pi}$ is the unique $\overline{\theta}$-generic submodule $\rl(\pi)$. In particular $\textup{dim}_{\overline{\mathbb{F}_\ell}}(\bar{\pi}_{\bar{\theta}})=1$.
\end{corollary}

\begin{proof} First of all there exist a parabolic subgroup $P = M U$ of $G$ and an irreducible supercuspidal representation $\overline{\rho}$ of $M$ such that $\bar{\pi}$ is a subquotient of $i_P^G(\bar{\rho})$. Thanks to Rodier's heredity of Whittaker models, which still holds for $\ell$-modular representations according to Section \ref{SecM1}, we deduce that $\bar{\rho}$ must be $\bar{\theta}|_{N_M}$-generic. By Proposition \ref{prop DHKM}, let $\rho$ be an irreducible integral $\ell$-adic representation of $M$ such that $r_\ell(\rho)$ contains $\bar{\rho}$.

We now want to prove that $\rho$ is $\theta$-generic. To do so, we use a reduction argument. Choose a $G$-stable $\overline{\mathbb{Z}}_\ell$-lattice $L$ in $\rho$. As $\rho \in \textup{Rep}_{\overline{\mathbb{Q}_\ell}}(G)$, the exactness of $\theta$-coinvariants in Lemma \ref{lemme exact2} gives:
$$L_{\theta} \subset \rho_\theta = \rho_{\theta_{\overline{\mathbb{Q}_\ell}}}.$$
The $\overline{\mathbb{Q}_\ell}$-vector space $\rho_\theta$ has dimension at most $1$ by uniqueness of $\ell$-adic Whittaker models. Furthemore we have, thanks to Lemma \ref{lemme Justin}, the equality of dimensions:
$$\textup{dim}_{\overline{\mathbb{F}_\ell}}(L_\theta \otimes_{\overline{\mathbb{Z}_\ell}} \overline{\mathbb{F}_\ell}) = \textup{dim}_{\overline{\mathbb{F}_\ell}}( (L \otimes_{\overline{\mathbb{Z}_\ell}} \overline{\mathbb{F}_\ell})_{\bar{\theta}}),$$
and Lemma \ref{lemme exact2} also imposes:
$$\textup{dim}_{\overline{\mathbb{F}_\ell}}( (L \otimes_{\overline{\mathbb{Z}_\ell}} \overline{\mathbb{F}_\ell})_{\bar{\theta}}) = \textup{dim}_{\overline{\mathbb{F}_\ell}}(r_\ell(\rho)_{\bar{\theta}}) \geq \textup{dim}_{\overline{\mathbb{F}_\ell}}(\bar{\rho}_{\bar{\theta}}) > 0.$$
So $L_\theta \neq 0$ and $\rho_\theta \simeq \overline{\mathbb{Q}_\ell}$.

Now, if $M$ is a sub-$\overline{\mathbb{Z}_\ell}$-module of $\overline{\mathbb{Q}_\ell}$, then $M \simeq \overline{\mathbb{Z}_\ell}$ if and only if its reduction modulo $\ell$ is non-zero. So $L_\theta \simeq \overline{\mathbb{Z}_\ell}$ and $\textup{dim}_{\overline{\mathbb{F}_\ell}}(r_\ell(\rho)_{\bar{\theta}}) = 1 = \textup{dim}_{\overline{\mathbb{F}_\ell}}(\bar{\rho}_{\bar{\theta}})$. Then we conclude by Rodier's heredity property that $\textup{dim}_{\overline{\mathbb{F}_\ell}}(\bar{\pi}_{\bar{\theta}}) = 1$. \end{proof}
 
\begin{rem} Note that the proof of Corollary \ref{cor DHKM} brings an alternative very short proof of Theorem 4.1 for $\ell$-modular representations. Only Rodier's heredity for $\ell$-modular
representations is used, as well as Rodier's heredity and multplicity at most one for $\ell$-adic
representations. \end{rem}

From Theorems \ref{Thm M1} and \ref{prop Yang}, together with Theorem \ref{multiplicities_main_thm} and Corollary \ref{cor DHKM} we deduce the following corollary, which is Proposition \ref{prop Yang} for $\ell$-modular representations:

\begin{corollary}\label{cor modular Yang}
For all irreducible $\ell$-modular $\overline{\theta}$-generic representation $\bar{\pi}$ of $G$, we have:
$$\overline{\pi}^\vee\simeq \overline{\pi}\circ I_{\theta}.$$
\end{corollary} 

\begin{proof}
By Corollary \ref{cor DHKM}, the representation $\overline{\pi}\circ I_{\theta}$ is the unique $\overline{\theta}^{-1}$-generic factor in $r_\ell(\pi \circ I_\theta)$. Now Proposition \ref{prop Yang} implies that $r_\ell(\pi \circ I_\theta) \simeq r_\ell(\pi^\vee)$. But the contragredient $\overline{\pi}^\vee$ is $\bar{\theta}^{-1}$-generic by Theorem \ref{multiplicities_main_thm} and it is a factor of $r_\ell(\pi^\vee)$ by \cite[I.9.7]{Vbook}. So the claim follows. 
\end{proof}

\bibliographystyle{plain}
\bibliography{Modlfactors}

\end{document}